\documentclass[11pt]{article}
\usepackage{srcltx}
\usepackage{eurosym}
\usepackage{mathtools}
\usepackage{amsmath}
\usepackage{amsfonts}
\usepackage{amssymb}
\usepackage{amsthm}
\usepackage{bm}
\usepackage{graphicx}
\usepackage{mathrsfs}
\usepackage{xcolor}
\usepackage{exscale}
\usepackage{color, soul}
\usepackage{latexsym}
\usepackage{authblk}
\usepackage{comment}
\allowdisplaybreaks

\usepackage[colorlinks,plainpages=true,pdfpagelabels,hypertexnames=true,colorlinks=true,pdfstartview=FitV,linkcolor=blue,citecolor=red,urlcolor=black]{hyperref}
\PassOptionsToPackage{unicode}{hyperref}
\PassOptionsToPackage{naturalnames}{hyperref}
\usepackage{enumerate}
\usepackage[shortlabels]{enumitem}
\usepackage{bookmark}
\usepackage{wasysym}
\usepackage{esint}
\usepackage[ddmmyyyy]{datetime}
\usepackage[margin=1in]{geometry}
\makeatletter
\g@addto@macro\normalsize{%
	\setlength\abovedisplayskip{4pt}
	\setlength\belowdisplayskip{4pt}
	\setlength\abovedisplayshortskip{4pt}
	\setlength\belowdisplayshortskip{4pt}
}
\numberwithin{equation}{section}
\everymath{\displaystyle}
\usepackage[capitalize,nameinlink]{cleveref}
\crefname{section}{Section}{Sections}
\crefname{subsection}{Subsection}{Subsections}
\crefname{condition}{Condition}{Conditions}
\crefname{hypothesis}{Hypothesis}{Conditions}
\crefname{assumption}{Assumption}{Assumptions}
\crefname{lemma}{Lemma}{Lemmas}
\crefname{definition}{Definition}{Definitions}

\crefformat{equation}{\textup{#2(#1)#3}}
\crefrangeformat{equation}{\textup{#3(#1)#4--#5(#2)#6}}
\crefmultiformat{equation}{\textup{#2(#1)#3}}{ and \textup{#2(#1)#3}}
{, \textup{#2(#1)#3}}{, and \textup{#2(#1)#3}}
\crefrangemultiformat{equation}{\textup{#3(#1)#4--#5(#2)#6}}%
{ and \textup{#3(#1)#4--#5(#2)#6}}{, \textup{#3(#1)#4--#5(#2)#6}}%
{, and \textup{#3(#1)#4--#5(#2)#6}}

\Crefformat{equation}{#2Equation~\textup{(#1)}#3}
\Crefrangeformat{equation}{Equations~\textup{#3(#1)#4--#5(#2)#6}}
\Crefmultiformat{equation}{Equations~\textup{#2(#1)#3}}{ and \textup{#2(#1)#3}}
{, \textup{#2(#1)#3}}{, and \textup{#2(#1)#3}}
\Crefrangemultiformat{equation}{Equations~\textup{#3(#1)#4--#5(#2)#6}}%
{ and \textup{#3(#1)#4--#5(#2)#6}}{, \textup{#3(#1)#4--#5(#2)#6}}%
{, and \textup{#3(#1)#4--#5(#2)#6}}

\crefdefaultlabelformat{#2\textup{#1}#3}

\newtheorem{theorem} {Theorem}[section]
\newtheorem{proposition} [theorem]{Proposition}

\newtheorem{lemma}[theorem]{Lemma}

\newtheorem{counter example}[theorem]{Counter Example}
\newtheorem{remark}[theorem] {Remark}
\newtheorem{definition}[theorem] {Definition}

\newtheorem{assumption}[theorem]{Assumption}

\def\CC{{\rm \kern.24em \vrule width.02em height1.4ex depth-.05ex \kern-.26emC}}

\def\TagOnRight

\def\AA{{it I} \hskip-3pt{\tt A}}

\def\QQ{\rlap {\raise 0.4ex \hbox{$\scriptscriptstyle |$}} {\hskip -0.1em Q}}

\makeatletter
\newcommand{\vo}{\vec{o}\@ifnextchar{^}{\,}{}}
\makeatother

\def\YYint#1#2#3{{\setbox0=\hbox{$#1{#2#3}{\iint}$}
		\vcenter{\hbox{$#2#3$}}\kern-.50\wd0}}

\def\XXint#1#2#3{{\setbox0=\hbox{$#1{#2#3}{\int}$}
		\vcenter{\hbox{$#2#3$}}\kern-.50\wd0}}

\makeatletter
\def\namedlabel#1#2{\begingroup
	\def\@currentlabel{#2}%
	\label{#1}\endgroup
}
\makeatother
\makeatletter
\newcommand{\rmh}[1]{\mathpalette{\raisem@th{#1}}}
\newcommand{\raisem@th}[3]{\hspace*{-1pt}\raisebox{#1}{$#2#3$}}
\makeatother

\newcounter{desccount}

\newcommand{\descref}[2]{\hyperref[#1]{\textnormal{\textcolor{black}{}\textcolor{blue}{ #2}\textcolor{black}{}}}}

\newcommand{\dref}[2]{\hyperref[#1]{\textcolor{black}{(}\textcolor{blue}{\bf #2}\textcolor{black}{)}}}
\newcommand{\be} {\begin{eqnarray}}
	\newcommand{\ee} {\end{eqnarray}}
\newcommand{\Bea} {\begin{eqnarray*}}
	\newcommand{\Eea} {\end{eqnarray*}}

\newcommand{\e}  {\epsilon}

\newcommand{\eps} {\epsilon}

\newcounter{whitney}
\refstepcounter{whitney}

\newcounter{ineqcounter}
\refstepcounter{ineqcounter}
\makeatletter
\def\ps@pprintTitle{%
	\let\@oddhead\@empty
	\let\@evenhead\@empty
	\def\@oddfoot{}%
	\let\@evenfoot\@oddfoot}
\makeatother
\usepackage[doublespacing]{setspace}
\usepackage[titletoc,toc,page]{appendix}

\makeatletter
\newcommand{\refcheckize}[1]{%
	\expandafter\let\csname @@\string#1\endcsname#1%
	\expandafter\DeclareRobustCommand\csname relax\string#1\endcsname[1]{%
		\csname @@\string#1\endcsname{##1}\wrtusdrf{##1}}%
	\expandafter\let\expandafter#1\csname relax\string#1\endcsname
}
\makeatother

\refcheckize{\cref}
\refcheckize{\Cref}


\makeatletter
\newcommand{\mainsectionstyle}{%
	\renewcommand{\@secnumfont}{\bfseries}
	\renewcommand\section{\@startsection{section}{2}%
		\z@{.5\linespacing\@plus.7\linespacing}{-.5em}%
		{\normalfont\bfseries}}%
}
\makeatother
\usepackage{pgf,tikz}
\usetikzlibrary{arrows}
\usetikzlibrary{decorations.pathreplacing}
\usepackage{enumitem}    

\usepackage{xpatch}
\xpatchcmd{\MaketitleBox}{\hrule}{}{}{}
\xpatchcmd{\MaketitleBox}{\hrule}{}{}{}

\date{}

\usepackage{scalerel}

\newcommand{\U}{\textbf{u}}

\makeatletter

\usepackage{subcaption}
\usepackage{hyperref}
\usepackage{caption}
\usepackage{subcaption}

\usepackage[utf8]{inputenc}
\allowdisplaybreaks

\DeclareCaptionSubType*[Alph]{figure}
\title{Global Existence for a Class of Keyfitz--Kranzer Systems with Application to Thin-Film Flows}

\author[1]{Rahul Barthwal\thanks{\href{mailto:rahul.barthwal@mathematik.uni-stuttgart.de}{Corresponding author: rahul.barthwal@mathematik.uni-stuttgart.de}}}
\author[2]{Philipp  \"Offner\thanks{\href{mailto: philipp.oeffner@tu-clausthal.de}{philipp.oeffner@tu-clausthal.de}}}
\author[1]{Christian Rohde\thanks{\href{mailto:christian.rohde@mathematik.uni-stuttgart.de}{christian.rohde@mathematik.uni-stuttgart.de}}}
\affil[1]{\footnotesize Institute of Applied Analysis and Numerical Simulation, University of Stuttgart, Germany}
\affil[2]{Institute of Mathematics, Clausthal University of Technology,  Germany} 

\begin{document}

\maketitle
\begin{abstract}
We prove the existence of global weak entropy solutions for a class of non-symmetric Keyfitz– Kranzer type systems that includes lubrication models for thin-film flow. 
We identify a family of entropy/entropy-flux pairs for these first-order systems, which is, in particular, admissible for a tailored second-order approximate system. The latter is motivated by higher-order dissipation operators in thin-film flow models. By identifying an invariant region in the state space, it is possible to derive a-priori $L^\infty$-bounds for the sequence of solutions to the approximate system. Exploiting the parabolic and transport structure of the equations associated with the Riemann invariants, we then rigorously justify the vanishing-diffusion limit and establish the existence of weak entropy solutions for the Cauchy problem for the first-order systems.
\end{abstract}
{\textbf{Key words.} Keyfitz--Kranzer systems, thin-film flow, vanishing diffusion limit, entropy solutions}
\medskip \\
{\textbf{MSC codes.}  35L40 35L45 35L65 35A01  35D30

\section{Introduction}
Let $T>0$, $n\in\mathbb N$ and consider the Cauchy problem
for $\mathbf{U}:\Omega_T:=\mathbb R\times(0,T)\to\mathcal U\subset\mathbb R^n$ given by
\begin{equation}\label{general_system}
\begin{array}{rcll}
\mathbf{U}_t+{\mathbf{F}(\mathbf{U})}_x&=&0& \text{ in } \Omega_T,\\
\mathbf{U}(\cdot,0)&=&\mathbf U_0& \text{ in } \mathbb R.
\end{array}
\end{equation}
In \eqref{general_system}, $\mathbf F:\mathcal U\to\mathbb R^n$ denotes the flux function and $\mathbf U_0:\mathbb R \to \mathcal U$ is the initial datum. We assume that \eqref{general_system} is hyperbolic, i.e., the Jacobian $\mathbf D\mathbf F$ of $\mathbf F$ is diagonalizable in the state space $\mathcal U$ and has real eigenvalues. Depending on $\mathbf F$, an appropriate weak solution concept is required for \eqref{general_system}.

The notion of weak entropy solutions refers to the existence of an entropy/entropy-flux pair $(\mathcal E,\mathcal Q)$ for \eqref{general_system}, i.e., the function  ${\mathcal E} \in C^2({\mathcal U}, \mathbb R)$ is supposed to be 
strictly convex and  ${\mathcal Q} \in C^1 ({\mathcal U}, \mathbb R)$  has to be chosen such that  the compatibility condition 
\renewcommand{\U}{{\mathbf U}}
\begin{equation}\label{compatibility conditions}
(\nabla_{\U} {\mathcal E})^\top {\mathbf D} {\mathbf F} = (\nabla_\U {\mathcal Q})^\top
\end{equation}
 holds in $\mathcal U$. A weak entropy solution is a weak solution 
of \eqref{general_system} that dissipates the entropy $\mathcal E$, see Definitions \ref{weak_soln} and \ref{entropy_soln} for precise formulations.\\ 
A standard approach to construct global weak entropy solutions is to introduce a second-order approximation of the quasilinear evolution equations in \eqref{general_system}, namely,
\begin{align}\label{general_viscous}
\mathbf{U}^\epsilon_t+{\mathbf{F}(\mathbf{U}^\epsilon)}_x
=
\epsilon\big(\mathbf{B}(\mathbf{U}^\epsilon)\mathbf{U}^\epsilon_x\big)_x
\qquad \text{ in } \Omega_T.
\end{align}
Here, $\epsilon>0$ is a small parameter and $\mathbf B:\mathcal U\to\mathbb R^{n\times n}$ is a suitably chosen diffusion matrix. The existence of a weak entropy solution for \eqref{general_system} is then established by deriving suitable estimates for the viscous approximations \eqref{general_viscous} and passing to the limit as $\epsilon\to0$.\\
The system \eqref{general_viscous} is regarded as an admissible approximation of \eqref{general_system} if it fits to the entropy structure of the underlying hyperbolic system \eqref{general_system}, in the sense that the entropy/entropy-flux pairs associated with \eqref{general_system} satisfy appropriate dissipation relations for the approximation system. Indeed, this admissibility is essential to derive the compactness properties required to pass to the limit $\epsilon\to 0$ in \eqref{general_viscous}.

For many hyperbolic systems, choosing the diffusion matrix as $\mathbf{B}(\mathbf{U})=\mathbf{I}$, $\mathbf I$ being the unit matrix, has proven to be sufficient to produce an admissible approximation; see e.g. \cite{bianchini_annals, diperna_cc, heibig_1994, Lu_SIMA}. This provides sufficient a-priori estimates for the approximate system and allows to pass to the limit $\epsilon\rightarrow 0$ to yield global weak entropy solutions of the system \eqref{general_system}. However, this is not the case for general systems. In particular, for systems with more delicate entropy structure, the specific form of the diffusion matrix $\mathbf{B}$ plays a crucial role in the construction of an admissible regularization.  In such situations, one often needs a diffusion matrix specifically tailored to the structure of the system rather than the standard unit matrix. An example is the work of Chen\&Perepelitsa \cite{chen2015vanishing} on the compressible Euler equations with spherical symmetry, where the vanishing viscosity approximation is built with designed viscosity terms, 
in order to establish convergence to global weak entropy solutions. To the best of the authors' knowledge, general existence results based on genuinely nonlinear or system-adapted diffusion matrices $\mathbf{B}$ remain relatively limited in comparison with the classical theory for  $\mathbf{B(U)}=\mathbf{I}$.

In this article, we study the Cauchy problem  for \eqref{general_system} with $n=2 $ for a class of first-order conservation laws of the form
\begin{equation}\label{eq: main_system}
\begin{aligned}
u_t+{\big(u\phi(r(u, v))\big)}_x&=0,\vspace{0.2 cm}\\[1.2ex]
   v_t+{\big(v\phi(r(u, v))\big)}_x&=0,
\end{aligned}\hspace{1 cm} 
\end{equation}
with initial data given by
\begin{align}\label{initial_data}
    (u, v)(\cdot, 0)=(u_0, v_0)\in \left(L^{\infty}(\mathbb{R})\right)^2.
\end{align}
In \eqref{eq: main_system} the coupling is realized in the product function 
\begin{equation}\label{r:def}
r(u, v)=uv.
\end{equation}
Moreover, $\phi:\mathbb{R}\mapsto \mathbb{R}$ is a smooth function of $r$ satisfying the following assumption.
\begin{assumption}\label{assumption}
$\phi\in C^\infty(\mathbb{R})$ with $\phi'(r)>0$ and $\text{{\rm {meas}}}\{r: \,3\phi'(r)+2r\phi''(r)=0\}=0$.
\end{assumption}
Furthermore, the state space for the two-dimensional unknown  $ {\mathbf U} = (u, v)^\top$ is defined for some constant $m>0$ as
\begin{align}\label{state_space}
   \mathcal{U}=\mathcal{U}_m=\{(u, v)^\top\in \mathbb{R}^2|~~ u, v\geq m\}.
\end{align}
The system \eqref{eq: main_system} turns out to be hyperbolic in ${\mathcal U}_m$. 
The  results of this paper stem from the discovery that we can provide  explicit Riemann invariants and a class of entropy/entropy-flux pairs 
for the system \eqref{eq: main_system} in ${\mathcal U}_m$, see Section \ref{sec: hyperbolicity}.

Systems of the form \eqref{eq: main_system} belong to the Keyfitz--Kranzer class of hyperbolic systems \cite{keyfitz1980system}. General  $(n\times n)$-Keyfitz--Kranzer systems  provide a  widely studied class of  hyperbolic systems that are given by
\begin{align}\label{eq: KK}
u_{it}+{(u_{i}(\Phi(u_1, u_2, \ldots, u_n)))}_x=0, ~i=1, 2, \ldots, n.
\end{align}
The systems of the form \eqref{eq: KK} have been analyzed in diverse settings for various choices of the function $\Phi$. Some of the particular choices of $\Phi$ include $\Phi({\mathbf U})=\phi(|{\mathbf U}|)$ \cite{freistuhler1991rotational, freistuhler1994cauchy} for $(n\times n$)-systems and  $\Phi(u_1, u_2)=\Phi(\alpha u_1+\beta u_2)$, $\Phi(u_1, u_2)=\Phi(u_1/u_2)$ \cite{yang2012new, yang2014delta}, or  $\Phi(u_1, u_2)=\Phi(u_1u_2)$ \cite{shen2018delta} for ($2\times2$)-systems. It is interesting to see that the different choices of $\Phi$ describe quite distinct physical phenomena. For instance, the choice of $\Phi({\mathbf U})=\phi(|{\mathbf U}|)$ describes a simplified magnetohydrodynamics model \cite{MR2205154, freistuhler1991rotational} or the elastic string model \cite{keyfitz1980system}, while $\Phi({\mathbf U})={k_i u_i}/{(1+ u_1 + \cdots + u_n})$ describes $n$-component chromatography equations \cite{james1995kinetic}.\\
For the ($2\times 2$)-system with $\Phi(u_1, u_2)=u_1u_2/2$, the system \eqref{eq: KK} or equivalently the system \eqref{eq: main_system} reduces to the system derived in \cite{barthwal2023construction, conn2017simple}, which governs the first-order dynamics of a thin film flow under the influence of a perfectly soluble solute; see also \cite{barthwal2022two, pandey2025construction}. 

It is remarkable to note that many of the Keyfitz--Kranzer type systems actually belong to the Temple systems \cite{temple1983systems}, where the shock and rarefaction curves coincide and form a straight line in the phase space. For Temple class systems, the total variation of Riemann invariants can be shown to be decreasing over time, and thus, the global existence of weak solutions with any initial data with bounded variation is possible using the classical $L^1$-theory.  In particular, Serre investigated Temple systems in this regard in \cite{serre_temple}. He also established the existence of a global unique weak solution for the case when the initial data is of bounded variation  \cite{serre_temple}. The condition of bounded variation initial data was further relaxed to a broader $L^\infty$-space for ($n\times n$)-Temple systems by Heibig \cite{heibig_1994}. By employing the compensated compactness framework, he was able to prove the existence of global weak solutions for $L^\infty$-initial data. Moreover, by constructing Oleinik-type estimates, the uniqueness of solutions in $BV(\mathbb{R})$ has been obtained. 

We emphasize that systems of the form \eqref{eq: main_system} do not belong to the Temple class, even though their shock and rarefaction curves coincide. Nevertheless, many Keyfitz-Kranzer type systems outside the Temple class have been studied successfully; see, for instance, the works of Lu \cite{Lu_JFA_1, Lu_JFA_2, Lu_SIMA} and the references therein. However, analysis for the system of the form \eqref{eq: main_system} remains open, even when the physical systems, including thin film systems, belong to this class. Most of the available results for the Keyfitz--Kranzer type systems rely on the standard approximation \eqref{general_viscous} with $\mathbf{B}=\mathbf{I}$, which yields a bounded invariant region for the set of Riemann invariants of the underlying first-order system and hence global solutions of the corresponding approximate system. However, this approach cannot be applied directly to systems of the form \eqref{eq: main_system}. More precisely, the maximum principle cannot be applied directly for the Riemann invariants associated with \eqref{eq: main_system}. Moreover, the entropy/entropy-flux pairs of the hyperbolic system \eqref{eq: main_system} do not lead to an admissible approximation \eqref{general_viscous} when $\mathbf{B}=\mathbf{I}$. As a result, the global existence of smooth solutions to the approximate system cannot be established in a straightforward manner. Therefore, for systems of the form \eqref{eq: main_system}, the choice $\mathbf{B}=\mathbf{I}$ does not appear to provide a suitable approximation.

In order to overcome this difficulty, we propose a novel approximation for the system \eqref{eq: main_system}, which takes the following form.
\begin{equation}\label{viscous_regularization}
\begin{aligned}
u^\eps_t+{\big(u^\eps\phi(r(u^\eps, v^\eps))\big)}_x&=\epsilon\left(\dfrac{{(u^\eps v^\eps)}_{x}}{v^\eps} \right)_x,\\[1.5ex]
   v^\eps_t+{\big(v^\eps\phi(r(u^\eps, v^\eps))\big)}_x&=\epsilon \left(\dfrac{{(u^\eps v^\eps)}_{x}}{u^\eps} \right)_x.
\end{aligned}
\end{equation}
The approximation system \eqref{viscous_regularization} is inspired by the lubrication equations for thin-film flows. To be precise, the second-order lubrication model for a thin-film flow under the influence of an anti-surfactant takes the form \cite{barthwal2023construction, barthwal2025hyperbolic, conn2017simple}
\begin{equation}\label{eq: thin_film_system}
\begin{aligned}
    h_t+\left(\dfrac{h^2b}{2}\right)_x&=0,\\
    b_t+\left(\dfrac{hb^2}{2}\right)_x&=\dfrac{1}{\mathrm{Pe}} \left(\dfrac{{(h b)}_{x}}{h} \right)_x.
\end{aligned}
\end{equation}
In \eqref{eq: thin_film_system}, the unknown   $h$ denotes the film thickness and $b$ stands for the unknown concentration gradient of the solute particles. Moreover, $\mathrm{Pe}$ denotes the P\'eclet number, which measures the ratio of advective to diffusive transport. In particular, $\mathrm{Pe}\ll 1$ corresponds to a diffusion-dominated regime, whereas $\mathrm{Pe}\gg 1$ corresponds to a convection-dominated regime.
In this spirit, identifying $1/\mathrm{Pe}=\epsilon$, the parameter $\epsilon$ in \eqref{viscous_regularization} can be interpreted as a physical diffusion coefficient. However, we note that the diffusion term in the first equation of \eqref{viscous_regularization} is an artificial parameter, which allows us to pass to the limit $\epsilon\rightarrow 0$ in \eqref{viscous_regularization}. 

We explicitly prove that the found entropies for the hyperbolic system \eqref{eq: main_system} are dissipated by the solutions of the approximate system \eqref{viscous_regularization}. 
Moreover, by utilizing a maximum principle on the Riemann invariants, we prove the existence of a bounded invariant region for the approximate system \eqref{viscous_regularization}, which induces global $L^\infty$-bounds for the solutions of \eqref{viscous_regularization}. 

The remainder of the article is structured as follows. In Section \ref{sec: hyperbolicity}, we discuss the basic properties of the system \eqref{eq: main_system}, identify the class of entropy/entropy-flux pairs, and explain how one can achieve local
wellposedness for the system \eqref{eq: main_system}.
Next, in Section \ref{sec: viscous}, we highlight the problems arising from using a classical diagonal diffusion matrix. In particular, the lack of the maximum principle avoids to deduct the existence of global solutions for the approximate system. To overcome this issue, we propose our novel approximation \eqref{viscous_regularization}  and deduce several a-priori bounds to prove the existence of global smooth solutions for the approximation system \eqref{viscous_regularization}. Finally, in Section \ref{sec: vanishing}, we utilize the a-priori bounds to ensure the existence of global weak entropy solutions
for \eqref{eq: main_system}. Concluding remarks and future outlook are provided in Section \ref{sec: conclusions}. 
\subsection*{Notations}\label{Notation}
Throughout the article, we use the following standard notation. We denote by $C^{k}(D)$ the space of functions on a domain $D$ whose derivatives up to order $k$ are continuous, where $k \in \mathbb{N}_0$. 
Further, for $p \in [1,\infty) \cup \{\infty\}$, we denote the Lebesgue space on $D$ as $L^p(D)$ and the $L^p(D)$-norm as ${\lVert \cdot \rVert}_{L^p(D)}$.\\
We denote by $H^k(D)$ the Sobolev space of order $k$ consisting of 
square-integrable functions whose weak derivatives up to order $k$ are also square-integrable, that is
\[
    H^k(D)
    := \bigl\{ g \in {L}^{2}(D) : \partial^m g \in {L}^{2}(D)
        \text{ for all multi-indices}~ m \text{ with } |m| \le k \bigr\},
\]
equipped with the norm
\[
    {\lVert g \rVert}_{H^k(D)}^2
    := \sum_{|m| \le k} {\lVert \partial^m g \rVert}_{{L}^{2}(D)}^2,
\]
where $\partial^m g$ denotes the weak derivative of $g$ of order $m$. \\
Finally, for a Banach space $X$ and $T>0$, we denote by $L^p(0, T;X)$, the Bochner space of measurable functions $g:(0,T)\to X$ with an induced norm
\[
    {\lVert g \rVert}_{L^p(0,T;X)}
    = \biggl( \int_0^T {\lVert u(t) \rVert}_{X}^p \,\mathrm{d}t \biggr)^{1/p},
\]
for $p \in [1,\infty)$, while for $p = \infty$
\[
{\lVert g \rVert}_{L^{\infty}(0,T;X)}
    = \sup_{t\in[0, T]} {\lVert g(t) \rVert}_{X}.
\]
\section{Hyperbolicity, entropy/entropy-flux pairs and local wellposedness}\label{sec: hyperbolicity}
In order to study the mathematical structure of system \eqref{eq: main_system}, we first show that it is hyperbolic. Assuming sufficient smoothness of the solutions, system \eqref{eq: main_system} can be expressed  for $\U= (u,v)^\top$ in the non-conservative form 
\begin{align}\label{eq: compact_form}
    \mathbf{U}_t+\mathbf{DF(U)} \U_x=0,
\end{align}
with  
\[
\mathbf{DF(U)}=\begin{pmatrix}
    \phi(r)+r\phi'(r) & u^2 \phi'(r)\\
    v^2 \phi'(r)& \phi(r)+r\phi'(r)
\end{pmatrix}.\]
being the Jacobian of the flux 
\[
\mathbf{F(U)}=\begin{pmatrix}
   u \phi(r)\\
   v \phi(r)
\end{pmatrix}
.\]
Note that we used $r = r(u,v)$ from \eqref{r:def}.\\
By a direct calculation, the eigensystem of the system \eqref{eq: compact_form} is given by 
\begin{align}\label{eigensystem}
\lambda_1(u, v) &= \phi(r),\quad ~\quad \quad \quad \quad \quad \mathbf{r}_1(u,v)=(-u, v)^\top , \\
\lambda_2(u, v) &= \phi(r)+2r\phi'(r), \quad\quad\mathbf{r}_2(u,v) = (u, v)^\top.
\end{align}
Noticing that $\phi'(r) > 0$, see Assumption \ref{assumption}, we have a clear ordering of the eigenvalues given by 
\[\lambda_1(u, v) < \lambda_2(u, v),~{\rm{when}}~(u, v)^\top\in \mathcal{U}_m.\] 
This implies that the system \eqref{eq: main_system} is strictly hyperbolic for all $\mathbf{U}\in \mathcal{U}_m$.\\
A straightforward computation leads us to 
\[\nabla_\mathbf{U} \lambda_1(u,v)\cdot \mathbf{r}_1(u,v) = 0, \qquad \nabla_\mathbf{U} \lambda_2(u,v)\cdot \mathbf{r}_2(u,v) = 6r\phi'(r) + 4r^2\phi''(r)\neq 0\quad \text{ in } \mathcal{U}_m\]
for functions $\phi$ satisfying Assumption \ref{assumption}.
This implies that the associated characteristic field of $\lambda_1$ is always linearly degenerate while that of $\lambda_2$ is genuinely nonlinear in  $\mathcal{U}_m$.

We  proceed to present explicit Riemann invariants for the system \eqref{eq: main_system} 
which pave the way to find entropy/entropy-flux pairs. 
In fact, we construct an entire family of entropy/entropy-flux pairs $(\mathcal{E}, \mathcal{Q})$ for the system \eqref{eq: main_system}. \\
First, by definition, Riemann invariants $w_i= w_i(u,v)$ are parallel to the left eigenvectors such that for $i\in \{1, 2\}$
\[\nabla_{\mathbf{U}} w_i(u,v)\cdot \mathbf{r}_i(u,v)=0.\]
Therefore, we  observe  that the Riemann invariants for the system \eqref{eq: main_system} are given by
\begin{align}\label{Riemann_invariants}
   w_1:=r(u,v)=uv, \qquad w_2:\xi(u,v):=\dfrac{u}{v}.
\end{align}
\begin{proposition}[Entropy/entropy-flux pairs for \eqref{eq: main_system}]\label{entropy_theorem}
Let  $\phi$ be a function satisfying Assumption \ref{assumption}. 
Consider functions $\Psi$ and $\Theta$ in  $C^1(0,\infty)$ such that 
\begin{align}\label{entropy_conditions}
\Psi''(r)>0, ~\Psi'(r)<0,~2r\Psi''(r)+\Psi'(r)>0, ~4\xi^2\Theta''(\xi)+4\xi\Theta'(\xi)-\Theta(\xi)\geq 0  
\end{align}
holds.\\
Then, the pair $(\mathcal{E}, \mathcal{Q})$ defined by
\begin{equation}\label{eq: entropyclass}
\begin{array}{rcl}
 \mathcal{E}(u, v)&=& \Psi(r)+\sqrt{r}\Theta\left(\xi\right) - \Psi(m^2)-m \Theta\left(1\right),\\[1ex]
 \mathcal{Q}(u, v)&=&\left(\displaystyle\int_{m^2}^{r} \Psi'(s)(\phi(s)+2s\phi'(s))\, ds\right)+\phi(r)\sqrt{r}\Theta(\xi)
\end{array}
\end{equation}
is an entropy/entropy-flux pair for the system \eqref{eq: main_system} in $\mathcal{U}_m$.
\end{proposition}
The constant term $ -\Psi(m^2)-m \Theta\left(1\right) $ in the entropy $\mathcal E$ in \eqref{eq: entropyclass} is convenient to consider later on well-defined space integrals of the entropy. 
\medskip 

{\textit{\hspace*{-0.6cm}Proof (of Proposition \ref{entropy_theorem}).}}
The Riemann invariants $r=uv$ and $\xi=u/v$ of the system \eqref{eq: main_system} satisfy the diagonal system
\begin{equation}\label{diagonal_system}
\begin{aligned}
    \begin{pmatrix}
        r \\ \xi
    \end{pmatrix}_t+ \begin{pmatrix}
        \phi(r)+2r \phi'(r) & 0 \\0 & \phi(r)
    \end{pmatrix} \begin{pmatrix}
        r \\ \xi
    \end{pmatrix}_x= \begin{pmatrix}
        0 \\ 0
    \end{pmatrix}.
    \end{aligned}
\end{equation}
We use the fact that  the mapping $(u,v) \mapsto (r,\xi)$ is invertible in ${\mathcal U}_m$ and consider, with  a slight  misuse of notation, the  entropy/entropy-flux pair 
$(\mathcal{E}, \mathcal{Q})$ as a function of the Riemann invariants. 
In view of the compatibility conditions \eqref{compatibility conditions}, any entropy/entropy-flux pair $(\mathcal{E}, \mathcal{Q}) = (\mathcal{E}, \mathcal{Q})(r, \xi)$ must satisfy
\begin{align*}
    2r\phi'(r)\mathcal{E}_{r\xi}-\phi'(r)\mathcal{E}_{\xi}=0.
\end{align*}
This implies by  Assumption \ref{assumption} that any entropy $\mathcal{E}$ should be of the form 
\begin{align}
    \mathcal{E}(r, \xi)=\Psi(r)+\sqrt{r}\Theta(\xi) + K, \quad K\in \mathbb{R},
\end{align}
along with an entropy flux satisfying
\begin{align}
     \mathcal{Q}(r, \xi)=\left(\displaystyle\int_{m^2}^r \Psi'(s)(\phi(s)+2s\phi'(s))\, ds\right)+\phi(r)\sqrt{r}\Theta(\xi).
\end{align}
In view of the diagonal system \eqref{diagonal_system}, a sufficient condition (see Section 7.4 in \cite{dafermos2005hyperbolic}) for $\mathcal{E}$ to be strictly convex for $\mathbf{U}\in \mathcal{U}_m$ is that 
\[\mathbf{r_i^\top} \nabla_{\mathbf{U}}^2 \mathcal{E}(r, \xi) \mathbf{r_i}>0 ~\text{for}~ i\in \{1, 2\},\]
where $\mathbf{r}_i$ are the eigenvectors of the system defined in \eqref{eigensystem} and $\nabla_{\mathbf{U}}^2$ is the Hessian operator w.r.t. $\mathbf{U}$. 
First, we use the Hessian identity
\begin{align}\label{hessian_identity}
 \nabla_{\mathbf{U}}^2 \mathcal{E}(r, \xi)=J^\top \nabla^2_{(r, \xi)} \mathcal{E}(r, \xi) J+\mathcal{E}_{r}\nabla_{\mathbf{U}}^2 r+\mathcal{E}_{\xi}  \nabla_{\mathbf{U}}^2 \xi, 
\end{align}
where $J$ denotes the Jacobian of $(r, \xi)$ w.r.t. $\mathbf{U}$.
A simple computation gives 
\[J\mathbf{r_1}=(0, -2\xi)^\top ~\text{and}~J\mathbf{r_2}=(2r, 0)^\top.\]
Therefore, for $i\in \{1, 2\}$, we multiply \eqref{hessian_identity} by $\mathbf{r_i^\top}$ from the left and by $\mathbf{r_i}$ from the right, and obtain 
\begin{align}\label{hessian_identity_new}
\mathbf{r_1^\top} \nabla_{\mathbf{U}}^2 \mathcal{E}(r, \xi)\mathbf{r_1}&=\sqrt{r}\left(4\xi^2\Theta''(\xi)+4\xi\Theta'(\xi)-\Theta(\xi)\right)-2r\Psi'(r),\\[1.1ex]
 \mathbf{r_2^\top}\nabla_{\mathbf{U}}^2 \mathcal{E}(r, \xi)\mathbf{r_2}&=2r(2r\Psi''(r)+\Psi'(r)).
\end{align}
Clearly, in view of the positivity of $r$ and $\xi$ in $\mathcal{U}_m$, a sufficient condition for the positivity of $\mathbf{r_1^\top} \nabla_{\mathbf{U}}^2\mathcal{E}(r, \xi)\mathbf{r_1}$ and $\mathbf{r_2^\top}\nabla_{\mathbf{U}}^2\mathcal{E}(r, \xi)\mathbf{r_2}$ is 
\[4\xi^2\Theta''(\xi)+4\xi\Theta'(\xi)-\Theta(\xi)=0,~ \Psi'(r)<0,  ~\Psi''(r)>0, ~2r\Psi''(r)+\Psi'(r)>0.\]
This concludes the proof of the theorem using $-K= \Psi(m^2)+ m\Theta\left(1\right)$.
\qed
\medskip 

In the remainder of the article, we consider special 
entropy/entropy-flux pairs. This includes the entropies
\begin{align}\label{entropy_main_convex}
    \mathcal{E}_{k, p}(u, v)=\dfrac{1}{(uv)^k}+\sqrt{uv}\left(\dfrac{u}{v}\right)^p -m^{-2k} -  m, \quad k>0, \quad p\geq \dfrac{1}{2}
\end{align}
with the associated entropy fluxes 
\begin{align}\label{flux_main_convex}
 {\mathcal Q}_{k, p}(u, v)=\displaystyle\int_{m^2}^r \left(\dfrac{-k\left(\phi(s)+2r\phi'(s)\right)}{s^{k+1}}\right)\, ds+\phi(r)\sqrt{r}\left(\dfrac{u}{v}\right)^p.
\end{align}
The pair $({\mathcal E}_{k, p}, {\mathcal Q}_{k, p})$ as defined in \eqref{entropy_main_convex}, \eqref{flux_main_convex} belongs to the general class of the entropy/entropy-flux pairs $(\mathcal{E}, \mathcal{Q})$ developed in Proposition \ref{entropy_theorem} by choosing \[
\Psi(\mathbf{U})=1/r^k \text{ and } \Theta(\mathbf{U})=\xi^p \text{ for }  k>0 \text{ and } p\geq 1/2.
\]
Moreover, the functions \eqref{entropy_main_convex} satisfy the sufficient conditions for convexity given in \eqref{entropy_conditions} and thus are strictly convex for $\mathbf{U}\in \mathcal{U}_m$. We also provide explicit calculations for the strict convexity of ${\mathcal E}_{k, p}$ in Lemma \ref{entropy_parabolic}. 

It is important to remark that one can choose different strictly convex entropies for the system \eqref{eq: main_system}. In particular, let us  choose $\Theta(\xi)=0$ and $\Psi(r)=\dfrac{1}{r^k},~k>0$. This gives the strictly convex entropy as 
\begin{align}\label{eq: r_entropy}
{\mathcal E}_k(u, v)={\mathcal E}_k(r)=\dfrac{1}{r^k} -m^{-2k}.
\end{align}
In view of the existence of entropy/entropy-flux pairs, the system \eqref{eq: main_system} is a symmetrizable hyperbolic system in the sense of Friedrichs, and thus the local well-posedness is a classical result; see e.g.\ \cite{dafermos2005hyperbolic}. Precisely, we have 
\begin{lemma}[Local wellposedness of the Cauchy problem for \eqref{eq: main_system}]
Let  initial data $\U_0$ with range in $ {\mathcal U}_m$ be given such that
\[
\U_0 - (m,m)^\top=  (u_0, v_0)^\top -  (m,m)^\top \in (H^s(\mathbb{R}))^2, \quad s > {3}/{2}.
\]
Then, there exists a  time $T^*=T^*\big(\lVert \U_0 -   (m,m)^\top \rVert_{H^s({\mathbb R})}\big) \in (0, T)$ for which the Cauchy problem for \eqref{eq: main_system} admits a unique classical solution $\bf U$ on $\Omega_{T^\ast}$.  
Moreover, the solution satisfies
\[\U-   (m,m)^\top=  (u, v)^\top -   (m,m)^\top \in C\left([0, T^*]; (H^s(\mathbb{R}))^2\right) 
\cap C^1\!\left([0, T^*]; (H^{s-1}(\mathbb{R}))^2\right).
\]
\end{lemma}
Our main interest is in weak entropy solutions with respect to the entropy/entropy-flux pairs $({\mathcal E}_{k, p}, { \mathcal Q}_{k, p})$ from Proposition \ref{entropy_theorem} as defined in \eqref{entropy_main_convex}.
\begin{definition}[Weak solution of the system \eqref{eq: main_system}]\label{weak_soln}
A function $\mathbf{U} \in L^\infty\big(\Omega_T;\,{\mathcal{U}}_m\big)$ is called a weak solution of the system \eqref{eq: main_system} with initial data 
$\mathbf{U}_0 \in L^\infty(\mathbb{R};\,{\mathcal{U}}_m)$, if we have for each vector-valued function $  {\bm \varphi} \in C_0^{1} \big( \Omega_T;\,\mathbb{R}^2\big)$ the identity 
\begin{equation*}
    \displaystyle\iint_{\Omega_T} 
    \Big( \mathbf{U}\cdot {\bm \varphi}_t
          + \mathbf{F}(\mathbf{U})\cdot {\bm \varphi}_x
         \Big)\,dx\,dt
    + \displaystyle\int_{\mathbb{R}} \mathbf{U}_0\cdot {\bm \varphi}(\cdot, 0)\,dx = 0.
\end{equation*}
\end{definition}

\begin{definition}[Weak entropy solution of the system \eqref{eq: main_system}]\label{entropy_soln}
A function $\mathbf{U}\in {L}^{\infty}(\Omega_T; ~\mathcal{U}_m )$ is called a weak  entropy solution of the system \eqref{eq: main_system}  with initial data 
$\mathbf{U}_0 \in L^\infty(\mathbb{R};\,{\mathcal{U}}_m)$ and associated with 
the entropy/entropy-flux pair  $({\mathcal E}_{k,p},{\mathcal Q}_{k,p}) $ defined in \eqref{entropy_main_convex}, \eqref{flux_main_convex}, if it is a weak solution and if 
\begin{align}\hspace*{-0.25cm}
   \displaystyle \iint_{\Omega_T} \Big(\mathcal{E}_{k, p}(\mathbf{U})\, \varphi_t+\mathcal{Q}_{k, p}(\mathbf{U})\,\varphi_x\Big) \, dx\, dt+\displaystyle\int_{\mathbb{R}} \mathcal{E}_{k, p}(\mathbf{U}_0)\,\varphi(\cdot, 0)\, dx\geq 0
\end{align}
holds for all non-negative functions $\varphi \in C^{1}_0( \Omega_T)$.
\end{definition}

\section{An entropy-admissible approximation of the  system \eqref{eq: main_system}}\label{sec: viscous}
In this section, we introduce a tailored approximation of \eqref{eq: main_system} that  
dissipates the entropies  $\mathcal{E}_{k, p}$ from  \eqref{entropy_main_convex}, i.e.,  an entropy dissipation relation holds for solutions of the approximate system. 

We begin by showing that the standard linear approximation \eqref{general_viscous} with $\mathbf{B}=\mathbf{I}$ is not suitable for \eqref{eq: main_system}. More precisely, under this choice, an invariant region for the variables $(u,v)^\top$ cannot be obtained directly from the maximum principle applied to the Riemann invariants. In addition, the entropy defined in \eqref{entropy_main_convex}  is not dissipated along solution trajectories of the approximate system. As a remedy, we introduce a nonlinear approximation adapted to the structure of \eqref{eq: main_system}.
\subsection{The approximation \eqref{general_viscous} with $\mathbf{B}=\mathbf{I}$}
For $\epsilon>0$, consider the approximate system
\begin{equation}\label{eq: main_system_viscos_1}
\begin{aligned}
u^\epsilon_t+\big(u^\epsilon\phi(r(u^\epsilon,v^\epsilon))\big)_x&=\epsilon u^\epsilon_{xx},\\
v^\epsilon_t+\big(v^\epsilon\phi(r(u^\epsilon,v^\epsilon))\big)_x&=\epsilon v^\epsilon_{xx}
\end{aligned}
\qquad \text{in } \Omega_T.
\end{equation}
The evolution equation for the first Riemann invariant $r^\epsilon:=u^\epsilon v^\epsilon$ using \eqref{eq: main_system_viscos_1} is then given by 
\begin{align*}
r^\epsilon_t
&=-(\phi(r^\epsilon)+2r^\epsilon\phi'(r^\epsilon))\,r^\epsilon_x
+\epsilon\bigl(r^\epsilon_{xx}-2u^\epsilon_xv^\epsilon_x\bigr),
\end{align*}
or equivalently,
\begin{equation}\label{eq:r-evolution}
r^\epsilon_t+(\phi(r^\epsilon)+2r^\epsilon\phi'(r^\epsilon))\,r^\epsilon_x
=\epsilon r^\epsilon_{xx}-2\epsilon u^\epsilon_xv^\epsilon_x.
\end{equation}
The additional term $-2\epsilon u^\epsilon_xv^\epsilon_x$ has no definite sign, and therefore the maximum principle cannot be applied directly to $r^\epsilon$.

Moreover, the convex entropies of the hyperbolic system \eqref{eq: main_system}, defined in \eqref{entropy_main_convex} or \eqref{eq: r_entropy}, are not dissipated along solutions of \eqref{eq: main_system_viscos_1}. Indeed, for the entropy
\begin{equation}\label{Ek}
{\mathcal E}_k(\mathbf U^\eps)=\frac{1}{(u^\eps v^\eps)^k} -m^{-2k}, \qquad k>0,
\end{equation}
one obtains from \eqref{eq: main_system_viscos_1} the identity
\[
{\mathcal E}_{k,t}+\big(\phi(r^\eps+2r^\eps\phi'(r^\eps)\big){\mathcal E}_x-2k\epsilon {\mathcal E}^{1+1/k}_ku^\eps_xv^\eps_x
=\epsilon\left({\mathcal E }_{k,xx}-\frac{(k+1){\mathcal E}_{k,x}^2}{k{\mathcal E}_k}\right).
\]
Because of the indefinite term $u_x^\eps v_x^\eps$, this identity does not yield a clear entropy dissipation mechanism. Therefore, the linear dissipation operator  \eqref{eq: main_system_viscos_1} does not appear to be suitable for the hyperbolic system \eqref{eq: main_system}.
\begin{remark}
The failure of the maximum principle at the level of the Riemann invariants is consistent with the fact that the system \eqref{eq: main_system} does not belong to the Temple class. In particular, it does not satisfy the criterion stated in Lemma 4.2 of \cite{serre_temple}. This can also be seen from the fact that the rarefaction curve associated with the first characteristic field is a hyperbola $\{uv=\mathrm{const.}\}$ in the $(u,v)$-plane rather than a straight line, even though the shock and rarefaction curves coincide; see \cite{barthwal2023construction, barthwal2025existence}.
\end{remark}
\subsection{A nonlinear approximation for the system \eqref{eq: main_system}}
We now introduce a nonlinear approximation adapted to the structure of \eqref{eq: main_system}. Motivated by the discussion above, the aim is to construct an approximate system for which both the invariant-region argument and the entropy structure are preserved. The proposed system is given by
\begin{equation}\label{eq: main_system_viscos_2}
\begin{aligned}
u^\epsilon_t+{\big(u^\epsilon\phi(u^\epsilon v^\epsilon)\big)}_x
&=\epsilon\left(\frac{(u^\epsilon v^\epsilon)_x}{v^\epsilon}\right)_x,\\[1.2ex]
v^\epsilon_t+{\big(v^\epsilon\phi(u^\epsilon v^\epsilon)\big)}_x
&=\epsilon\left(\frac{(u^\epsilon v^\epsilon)_x}{u^\epsilon}\right)_x,
\end{aligned}
\qquad \text{in } \Omega_T.
\end{equation}
We propose the  initial  condition 
\[
\U^\eps(\cdot, 0) =
(u^\eps(\cdot,0), v^\eps(\cdot,0))^\top= \U^\eps_0 \text{ in } \mathbb R,
\] 
with $\U^\eps_0$  being the Friedrichs convolution $\U^\eps = j^\eps * \U_0 $.
Here 
$j^\eps$ is a mollifier satisfying 
\[
j^\eps(x)=\dfrac{1}{\eps}j\left(\dfrac{x}{\eps}\right),
\]
with
\[
j(x)=\begin{cases}
    \dfrac{1}{A}\exp{\bigg\{\dfrac{1}{|x|^2-1}\bigg\}}, \quad |x|<1,\\
    0, \quad \quad \quad \qquad \quad\qquad~ |x|\geq 1,
\end{cases}
\]
and  $A$ chosen such that $\displaystyle\int_{-1}^{1}j(x)\, dx=1$.   One can verify that ${\{\U^\eps_0\}}_{\eps >0} \in (C^\infty({\mathbb{R}}),{\mathcal U}_m)^2 \cap (L^\infty(\mathbb{R}))^2$  satisfies for $\eps \to 0$ (see \cite{Evans})
\[
\U_0^\eps\rightarrow \U_0\,\, \text{a.e.} \, \text{and in }\, (L^p_{loc}(\mathbb{R}))^2 \text{ for }  p \in [1,\infty).
\]
In what follows, we prove that the class of convex entropies ${\mathcal E}_{k,p}$ as defined in \eqref{entropy_main_convex} is dissipated by the solutions of the system \eqref{eq: main_system_viscos_2}. We further prove that the Cauchy problem for the system \eqref{eq: main_system_viscos_2} is globally well-posed for smooth solutions. We start with the entropy dissipation statement that includes an a-priori bound on the spatial derivative of $r^\eps := u^\eps v^\eps$.
\begin{lemma}\label{entropy_parabolic}
For any $k>0$ and $p\geq 1/2$, the entropy 
\[
{\mathcal E}_{k, p}(u,v) = \frac{1}{(u v)^k}+\sqrt{u v}\left(\dfrac{u}{v}\right)^p - m^{-2k} - m
\]
is dissipated by the smooth solutions of the approximate system \eqref{eq: main_system_viscos_2}. In particular, for smooth solutions $\U^\eps$ with sufficient decay, the entropy identity 
\begin{equation}\label{entrop_h_p_general}
\int_{\mathbb{R}} {\mathcal E}_{k,p}(\mathbf{U}^\eps(x,t))\,dx
+ \epsilon\,\int_0^t\!\!\int_{\mathbb{R}}
\dfrac{k(2k+1)}{(r^{\epsilon})^{k+2}}\big(  r^{\epsilon}_x\big)^2\,dx\,ds
= \int_{\mathbb{R}} {\mathcal E}_{k,p}(\mathbf{U}_0^\eps(x))\,dx
\end{equation}
holds for $t\in [0,T]$ and 
${\mathcal E}_{k,p}(\mathbf{U}_0) \in L^1(\mathbb{R})$. 
\end{lemma}

\begin{proof}
Note that the approximate system \eqref{eq: main_system_viscos_2} can be written in the compact form \eqref{general_viscous} with
\[
\mathbf{B}(\mathbf{U}^\eps) 
= \begin{pmatrix}
1 & \dfrac{u^\eps}{v^\eps}\\[4pt]
\dfrac{v^\eps}{u^\eps} & 1
\end{pmatrix}.
\]
Now, we compute the Hessian matrix of ${\mathcal E}_{k,p}(\mathbf{U}^\eps)$ as
\begin{align*}
\nabla_{\mathbf{U}^\eps}^2{\mathcal E}_{k,p}(u^\eps,v^\eps)
&=
\nabla_{\mathbf{U}^\eps}^2 \left(\dfrac{1}{(u^\eps v^\eps)^k}\right)+\nabla_{\mathbf{U}^\eps}^2 \left((u^\eps)^{p+1/2}(v^\eps)^{1/2-p}\right).
\end{align*}
The Hessian of ${(u^\eps v^\eps)}^{-k}$ is given by 
\[
\nabla_{\mathbf{U}^\eps}^2 \left(\dfrac{1}{(u^\eps v^\eps)^k}\right)=\begin{pmatrix}
\dfrac{k(k+1)}{({u^\eps})^{k+2}({v^\eps})^k} & \dfrac{k^2}{(u^\eps)^{k+1} (v^\eps)^{k+1}}\\[6pt]
\dfrac{k^2}{({u}^\eps)^{k+1}({v}^\eps)^{k+1}} & \dfrac{k(k+1)}{({u}^\eps)^k ({v}^\eps)^{k+2}}
\end{pmatrix},
\]
with determinant 
\[
\det\left(\nabla_{\mathbf{U}^\eps}^2 \left(\dfrac{1}{(u^\eps v^\eps)^k}\right)\right)
= \frac{k^2(2k+1)}{({u}^\eps)^{2k+2}({v}^\eps)^{2k+2}} > 0,
\]
and positive diagonal entries for $k>0$ for $\mathbf{U}^\eps\in {\mathcal{U}}_m$. Hence, $\nabla_{\mathbf{U}^\eps}^2 \left(\dfrac{1}{(u^\eps v^\eps)^k}\right)$ is positive definite for $\mathbf{U}^\eps\in \mathcal{U}_m$. 

Similarly, the Hessian of $\left((u^\eps)^{p+1/2}(v^\eps)^{1/2-p}\right)$ is given by 
\[
\nabla_{\mathbf{U}^\eps}^2 \left((u^\eps)^{p+1/2}(v^\eps)^{1/2-p}\right)=\left(p^2-\dfrac{1}{4}\right)
\begin{pmatrix}
\dfrac{1}{u^\eps}\left(\dfrac{u^\eps}{v^\eps}\right)^{p-1/2} & -\dfrac{1}{v^\eps}\left(\dfrac{u^\eps}{v^\eps}\right)^{p-1/2}\\[1.4ex]
-\dfrac{1}{v^\eps}\left(\dfrac{u^\eps}{v^\eps}\right)^{p-1/2} & \dfrac{1}{v^\eps}\left(\dfrac{u^\eps}{v^\eps}\right)^{p+1/2} 
\end{pmatrix},
\]
with determinant 
\[
\det\left(\nabla_{\mathbf{U}^\eps}^2  \left((u^\eps)^{p+1/2}(v^\eps)^{1/2-p}\right)\right)= 0
\]
and positive diagonal entries for $\mathbf{U}^\eps\in \mathcal{U}_m$. Hence, $\nabla_{\mathbf{U}^\eps}^2 \left((u^\eps)^{p+1/2}(v^\eps)^{1/2-p}\right)$ is positive semi-definite for $\mathbf{U}^\eps\in \mathcal{U}_m$ and $p>1/2$. This implies that the Hessian matrix $\nabla_{\mathbf{U}^\eps}^2{\mathcal E}_{k,p}(u^\eps,v^\eps)$ is positive definite and thus the entropy ${\mathcal E}_{k, p}$ is strictly convex for $\mathbf{U}^\eps\in \mathcal{U}_m$.

Next, we consider a classical solution $\mathbf{U}^\eps = (u^\eps,v^\eps)^\top$ of \eqref{eq: main_system_viscos_2}.  Multiplying \eqref{eq: main_system_viscos_2} by $\nabla_\U {\mathcal E}_{k,p}(\mathbf{U}^\eps)$, we have
\[
{{\mathcal E}_{k,p}(\mathbf{U}^\eps)}_t + {{\mathcal Q}_{k,p}(\mathbf{U}^\eps)}_x
= \epsilon\,\nabla_\U {\mathcal E}_{k,p}(\mathbf{U}^\eps) \cdot {\big(\mathbf{B}(\mathbf{U}^\eps)\mathbf{U}^\eps_x\big)}_x,
\]
where ${\mathcal Q}_{k,p}$ is the entropy flux of the inviscid system as defined in \eqref{flux_main_convex}. Thus, after simplification, we obtain
\[
\nabla_{\U^\eps} {\mathcal E}_{k,p}(\mathbf{U}^\eps) \cdot {\big(\mathbf{B}(\mathbf{U}^\eps)\mathbf{U}^\eps_x\big)}_x
= {\big( \nabla_\U {\mathcal E}_{k,p}(\mathbf{U}^\eps) \mathbf{B}(\mathbf{U}^\eps)\mathbf{U}^\eps_x \big)}_x
- {\mathbf{U}^\eps_x}^\top \nabla_{\mathbf{U}^\eps}^2{\mathcal E}_{k,p}(\mathbf{U}^\eps) \mathbf{B}(\mathbf{U}^\eps)\mathbf{U}^\eps_x.
\]
Thus,
\begin{equation}\label{entropy_relation_general}
{{\mathcal E}_{k,p}(\mathbf{U}^\eps)}_t + {{\mathcal Q}_{k,p}(\mathbf{U}^\eps)}_x
= \epsilon\,{\big( \nabla {\mathcal E}_{k,p}(\mathbf{U}^\eps) \mathbf{B}(\mathbf{U}^\eps)\mathbf{U}^\eps_x \big)}_x
- \epsilon\,{\mathbf{U}^\eps_x}^\top \nabla_{\mathbf{U}^\eps}^2{\mathcal E}_{k,p}(\mathbf{U}^\eps) \mathbf{B}(\mathbf{U}^\eps)\mathbf{U}^\eps_x.
\end{equation}
A direct calculation shows
\begin{equation}\label{computeresult}
({\mathbf{U}^\eps_x})^\top \nabla_{\mathbf{U}^\eps}^2{\mathcal E}_{k,p}(\mathbf{U}^\eps) \mathbf{B}(\mathbf{U}^\eps)\mathbf{U}^\eps_x
= \frac{k(2k+1)}{({r^{\epsilon}})^{k+2}}(r^{\epsilon}_x)^2,
\end{equation}
where $r^{\epsilon}_x = \partial_x(u^\eps v^\eps)$. 

Therefore,
\begin{align}\label{entropy_main_identity}
{{\mathcal E}_{k,p}(\mathbf{U}^\eps)}_t + {{\mathcal Q}_{k,p}(\mathbf{U}^\eps)}_x
= \epsilon\,{\big( \nabla {\mathcal E}_{k,p}(\mathbf{U}^\eps) \mathbf{B}(\mathbf{U}^\eps)\mathbf{U}^\eps_x \big)}_x
- \epsilon\,\frac{k(2k+1)}{({r^{\epsilon}})^{k+2}}(r^{\epsilon}_x)^2.
\end{align}
Integrating in space over $\mathbb{R}$ and in time over $(0,t)$, and assuming  sufficient decay  so that the boundary flux term vanishes, we obtain
\[
\int_{\mathbb{R}} {\mathcal E}_{k,p}(\mathbf{U}^\eps(x,t))\,dx
+ \epsilon\,\int_0^t\!\!\int_{\mathbb{R}}
\dfrac{k(2k+1)}{(r^{\epsilon})^{k+2}}\big( r^{\epsilon}_x\big)^{2}\,dx\,ds
= \int_{\mathbb{R}} {\mathcal E}_{k,p}(\mathbf{U}^\eps_0(x))\,dx < \infty.
\]
\end{proof}
Lemma \ref{entropy_parabolic} motivates us to define an entropy-admissible solution for the approximate system \eqref{eq: main_system_viscos_2}. First, by multiplying the identity \eqref{entropy_main_identity} with a test function $\varphi \in C^{1}_0(\mathbb{R} \times [0,T))$ and integrating, we obtain the identity
\begin{align*}
   \displaystyle \iint_{\Omega_T} \Big(\mathcal{E}_{k, p}(\mathbf{U}^\eps)\, \varphi_t+\mathcal{Q}_{k, p}(\mathbf{U}^\eps)\,\varphi_x\Big) \, dx\, dt&+\displaystyle\int_{\mathbb{R}} \mathcal{E}_{k, p}(\mathbf{U}_0^\eps)\,\varphi(\cdot, 0)\, dx\\
   &\geq \epsilon \displaystyle \iint_{\Omega_T} \Big(\nabla \mathcal{E}_{k, p}(\mathbf{U}^\eps) \mathbf{B(U^\eps)} \mathbf{U}_x^\epsilon\Big)\, \varphi_x \, dx\,dt. 
\end{align*}
And thus, we define the entropy admissible solutions for the system \eqref{eq: main_system_viscos_2} associated with the entropy $\mathcal{E}_{k, p}$ as follows.
\begin{definition}[Entropy-admissible solutions of the system \eqref{eq: main_system_viscos_2}]\label{admissible_soln_viscous}
For a fixed $\eps>0$, a classical solution  $\mathbf{U^\epsilon}\in C^{2,1}(\Omega_T; ~\mathcal{U}_m )$ is called an entropy-admissible solution of the system \eqref{eq: main_system_viscos_2} associated with a convex entropy $\mathcal{E}_{k, p}$ satisfying \eqref{compatibility conditions}, if $\nabla_{\mathbf{U}}^2 \mathcal{E}_{k, p} \mathbf{B}$ is a symmetric positive semi-definite matrix in ${\mathcal U}_m$ and if the identity
\begin{equation}\label{definition_entropy_viscous}
\begin{aligned}
   \displaystyle \iint_{\Omega_T} \Big(\mathcal{E}_{k, p}(\mathbf{U}^\eps)\, \varphi_t+\mathcal{Q}_{k, p}(\mathbf{U}^\eps)\,\varphi_x\Big) \, dx\, dt&+\displaystyle\int_{\mathbb{R}} \mathcal{E}_{k, p}(\mathbf{U}_0^\eps)\,\varphi(\cdot, 0)\, dx\\
   &\geq \epsilon \displaystyle \iint_{\Omega_T} \Big(\nabla \mathcal{E}_{k, p}(\mathbf{U}^\eps) \mathbf{B(U^\eps)} \mathbf{U}_x^\epsilon\Big)\, \varphi_x \, dx\,dt 
\end{aligned}
\end{equation}
holds for all non-negative functions $\varphi \in C^{1}_0(\mathbb{R} \times [0,T))$.
\end{definition}
In view of the definition of entropy-admissible solutions for the system \eqref{eq: main_system_viscos_2} and the weak entropy solutions of the system \eqref{eq: main_system}, it is then clear that the sequence of entropy-admissible solutions of system \eqref{eq: main_system_viscos_2} converges to a weak entropy solution of the system \eqref{eq: main_system} for $\epsilon\rightarrow 0$, provided the strong limit $\mathbf{U}^\epsilon\rightarrow \mathbf{U}$ exists in $L^\infty(\Omega_T)$. 

\subsubsection{A bounded invariant region and global wellposedness of smooth solutions for the approximate system \eqref{eq: main_system_viscos_2}}
In this section, we prove uniform a-priori $L^\infty$-bounds for classical solutions  $\U^\eps$ of the Cauchy problem for 
\eqref{eq: main_system_viscos_2}. In particular, we  consider a transformed system for the 
Riemann invariants, which allows to apply the parabolic maximum principle for one of the evolution equations. This helps us to obtain a bounded invariant region for solutions of the original system \eqref{eq: main_system_viscos_2} and to construct classical solutions up to any time $T\in (0, \infty)$. 
\begin{lemma}[Invariant region for the approximate system \eqref{eq: main_system_viscos_2}]\label{invariant_region}
Let $M > m $ be a constant such that the initial datum 
$\U_0^\eps$ satisfies
\begin{align}\label{initial_invariant}
(u_0, v_0)\in [m, M]^2 \subset {\mathcal U}_m.
\end{align}
Then, the  set  $[m, M]^2$ is a bounded invariant region for the system \eqref{eq: main_system_viscos_2}, i.e., for any classical solution $\U^\eps$ of the Cauchy problem in 
$\Omega_T$ we have
\begin{equation}\label{inftybound}
\U^\eps(x,t)\in [m, M]^2 
\end{equation}
for all $(x, t)\in \Omega_T$.
\end{lemma}

\begin{proof} For classical solutions in a bounded domain, the  viscous approximation \eqref{eq: main_system_viscos_2} can be written equivalently in terms of the Riemann invariants, 
see \eqref{Riemann_invariants}.  The equations are given by
\begin{align}
r^{\epsilon}_t+(\phi(r^{\epsilon})+2r^{\epsilon}\phi'(r^{\epsilon}))r^{\epsilon}_x&=\epsilon \left(2r^{\epsilon}_{xx}-\dfrac{{(r^{\epsilon}_x)}^2}{r^{\epsilon}}\right), \label{eq: r_equation}\\
    \xi^\eps_t+\left(\phi(r^{\epsilon})-\eps \dfrac{r^{\epsilon}_x}{r^\eps}\right) \xi^\eps_x&=0, \label{eq: xi_equation}
\end{align}
where $r^{\epsilon}=u^\eps v^\eps$ and $\xi^\eps=u^\eps/v^\eps$. Note that equation \eqref{eq: r_equation} is an equation for $r^\eps$ only.

Since the initial datum satisfies $u_0, v_0 \in [m,M]$, we have
\[
r^\eps(\cdot,0) \in [m^2,M^2],
\qquad
\xi^\eps (\cdot,0) \in \left[\frac{m}{M},\frac{M}{m}\right].
\]
We consider  the change of variables $r^{\epsilon}=(\theta^\eps)^2$ such that \eqref{eq: r_equation} for $\theta^\eps$ becomes
\begin{equation}\label{theta}
\theta^\eps_t+\left(\phi((\theta^\eps)^2)+2 (\theta^\eps)^2\phi'((\theta^\eps)^2)\right)\theta^\eps_x=2 \epsilon \theta_{xx}.    
\end{equation}
Moreover,
\[
m\le \theta^\epsilon(x,0)\le M \quad \forall\, x\in \mathbb{R}.
\]
Since \eqref{theta} is a scalar  parabolic equation with diffusion coefficient $2\epsilon>0$, the parabolic maximum principle applies and yields
\[
m\le \theta^\epsilon(x,t)\le M,
\qquad \forall \,(x,t)\in \Omega_T.
\] 
Consequently,
\begin{equation}\label{eq: r_bounds}
m^2 \le r^\eps(x,t) \le M^2\quad \forall\, (x, t)\in \Omega_T.
\end{equation}
Next, \eqref{eq: xi_equation} is a linear transport equation for $\xi^\epsilon$ which can be written in the form
\begin{equation}\label{transport}
\xi^\epsilon_t+b^\epsilon(x,t)\xi^\epsilon_x=0,
\qquad
b^\epsilon=\phi(r^\epsilon)-\epsilon \frac{r_x^\epsilon}{r^\epsilon}.
\end{equation}
Since $(u^\epsilon,v^\epsilon)$ is smooth on its interval of local existence, so is $r^\epsilon$, and by \eqref{eq: r_bounds} we have $r^\epsilon\ge m^2>0$. Hence, the coefficient $b^\epsilon$ is smooth, and the method of characteristics applies. Along characteristics, $\xi^\epsilon$ is constant, and thus
\begin{equation}\label{eq: xi_bounds}
\frac{m}{M}\le \xi^\epsilon(x,t)\le \frac{M}{m}
\qquad \forall \,(x,t)\in \Omega_T.
\end{equation}
Therefore, in view of the inverse transformation 
\[
u^\eps=\sqrt{r^{\epsilon} \xi^\eps}, \qquad v^\epsilon=\sqrt{\dfrac{r^{\epsilon}}{\xi^\eps}},
\]
and combining the bounds for $r^\eps$ and $\xi^\eps$ from \eqref{eq: r_bounds} and \eqref{eq: xi_bounds}, we conclude
\[
m \le u^\eps(x,t) \le M,
\qquad
m \le v^\eps(x,t) \le M\quad \forall\, (x, t)\in \Omega_T.
\]
Therefore, the rectangle $[m,M]^2$ is a bounded invariant region.
\end{proof}

With the  $L^\infty$-bounds that have been derived in the proof of   Lemma \ref{invariant_region}, we can establish the existence of entropy-admissible weak, actually classical, solutions of the Cauchy problem for the system \eqref{eq: main_system_viscos_2} in the sense of Definition \ref{admissible_soln_viscous}.

\begin{lemma}[Global wellposedness of the Cauchy problem  for  \eqref{eq: main_system_viscos_2}]\label{global_visocus}Let $M > m $ be a constant such that the initial datum 
$\U_0$ satisfies
\begin{align}\label{initial_invariant}
(u_0, v_0)\in [m, M]^2 \subset {\mathcal U}_m 
\end{align}
Furthermore, we suppose that the derivatives of the functions  $r_0^\eps = u_0^\eps v_0^\eps$ 
and $\xi_0^\eps = u_0^\eps/v_0^\eps$  are bounded up to order three.\\
Then,  the Cauchy problem for 
\eqref{eq: main_system_viscos_2} admits a unique classical  solution $\U^\eps$ 
in $\Omega_T$.\\ 
Moreover, for each $\eps>0$ and  ${\mathcal E}_{k,p}(\mathbf{U}_0^\eps) \in L^1(\mathbb{R})$, the classical solution $\mathbf{U}^\eps$ satisfies the entropy dissipation relation \eqref{entrop_h_p_general}, i.e., it is entropy-admissible in the sense of Definition \ref{admissible_soln_viscous}.
\end{lemma} 

\begin{proof}
We note that classical solutions of the transformed Cauchy problem \eqref{eq: r_equation}, \eqref{eq: xi_equation}  induce  
classical solutions for the  Cauchy problem for \eqref{eq: main_system_viscos_2} and vice 
versa since the mapping $(u,v) \mapsto (r,\xi)$ is 
bijective and smooth on $[m,M]^2$.\\
We consider the Cauchy problem for  \eqref{eq: r_equation}, \eqref{eq: xi_equation} and observe first  that \eqref{eq: r_equation} and \eqref{theta}
are equivalent. The   Cauchy problem for \eqref{theta} is uniquely solvable for $\theta^\eps = \sqrt{r^\eps}$ and from the proof of Lemma \ref{invariant_region} we get thus by the solution theory for scalar parabolic equations and the parabolic maximum principle the wellposedness of classical solutions $r^\eps $ of \eqref{eq: r_equation} with $r^\eps(x,t) \in [m^2,M^2]$ for all $(x,t) \in \Omega_T$. Moreover, the bounds on the initial datum $r_0^\eps$ imply that the derivatives of $r^\eps(\cdot,t) $  are bounded up to third order in $\bar\Omega_T$.\\
We turn to \eqref{eq: xi_equation}, which can be written as a linear, first-order transport equation, see \eqref{transport}. The transport coefficient $b^\eps$ in \eqref{transport} is smooth and bounded in $\bar \Omega_T$ because it depends only on the 
quantities $r^\eps$ and $r^\eps_x$. Therefore, the method of characteristics 
implies the existence of a unique classical solution 
of the Cauchy problem for \eqref{transport}
which satisfy the bound \eqref{eq: xi_bounds}.
Higher-order regularity in space follows from the regularity of $r^\eps$ and the bounds on the derivatives of  $r^\eps(\cdot,t) $.\\
Altogether, we have verified the well-posedness of classical solutions to the Cauchy problem for \eqref{eq: main_system_viscos_2} in $ [m, M]^2$. It is straightforward to check that \eqref{entrop_h_p_general} holds.
\end{proof}
\section{The vanishing diffusion limit for the system \eqref{eq: main_system_viscos_2}}\label{sec: vanishing} 
In this section, we pass to the limit $\epsilon \rightarrow 0$ in the system \eqref{eq: main_system_viscos_2} and to prove the existence of global weak entropy solutions of the Cauchy problem for the hyperbolic system \eqref{eq: main_system}. Precisely, for the system \eqref{eq: main_system} with initial data \eqref{initial_data}, we have the following global existence result as the main result of this article.
\begin{theorem}[Main result]\label{main_theorem_1}
Let $M > m>0 $ be constants such that the initial datum 
$\U_0$ satisfies
\begin{align}\label{initial_invariant}
(u_0, v_0)\in [m, M]^2 \subset {\mathcal U}_m. 
\end{align}
Furthermore, we suppose that we have ${\mathcal E}_{k,p}(\mathbf{U}_0) \in L^1(\mathbb{R})$.
\\
Then the following statements hold true for the sequence ${\{ {\U^\eps}\}}_{\eps >0}$ of  classical solutions  of the Cauchy problem for \eqref{eq: main_system_viscos_2} (see Lemma \ref{global_visocus}).
\begin{enumerate}
\item Under Assumption \ref{assumption},
there exists a subsequence of ${\{r^{\epsilon}=u^\epsilon v^\epsilon\}}_{\epsilon>0}$ (still labelled ${\{r^{\epsilon}\}}_{\epsilon>0}$) and 
a function $ r\in L^\infty(\Omega_T)$, such that $r^\epsilon \rightarrow  r$ holds a.e.\ for  $\epsilon\rightarrow 0$.
\item Assume in addition that there is is a constant $\Xi>0$  such that the ratio  $\xi_0 = {u_0}/{v_0}$ satisfies
\[
\displaystyle\int_{\mathbb{R}} \left| \xi_{0,x}\right|\, dx\leq \Xi.
\]
Then, there exists a subsequence of
${\{\U^\eps\}}_{\eps >0}$ (still denoted
by ${\{\U^\epsilon\}}_{\eps>0}$) and a function $\U = (u,v)^T\in L^\infty(\Omega_T,{\mathcal U}_m)$, such that $\U^\eps \to \U$ pointwise a.e.\ and 
such that we have $r =  u v$ a.e.\\
The limit $\U$ is a weak entropy solution of the Cauchy problem for the hyperbolic system \eqref{eq: main_system} in the sense of Definition \ref{entropy_soln}. 
\end{enumerate}
\end{theorem}
\begin{proof}
In what follows, we prove that $\mathbf{U}^\eps\rightarrow \mathbf{U}$ strongly in $L^p_{loc}(\Omega_T), \, \, p\in [1, \infty)$. To this end, we utilize the Riemann invariant structure in the equivalent system  \eqref{eq: r_equation}, \eqref{eq: xi_equation}. In particular, using the transformation $\theta^\eps=\sqrt{r^\eps}$, it is easy to see that $\theta^\eps$ satisfies a viscous scalar balance law with a flux function satisfying Assumption \ref{assumption}, and thus using the classical compensated compactness theory of scalar viscous conservation laws (see Theorem 3.1.3 in \cite{lu2002hyperbolic}),  there is a function $\theta \in L^\infty(\Omega_T;[m,M] )$ and a subsequence still denoted as ${\{\theta^\eps\}}_{\eps >0}$ such that $\theta^\eps\rightarrow \theta$ a.e.\ in $\Omega_T$.
Hence, by the dominated convergence theorem, one can conclude that
\[
\theta^\epsilon\rightarrow \theta \quad \mathrm{as}~ \e\rightarrow 0~\quad \mathrm{in}\,\, L^p_{\mathrm{loc}}(\Omega_T),~p\in [1, \infty).
\]
Using the inverse transformation $r^\eps=(\theta^\eps)^2$, we then have the convergence of the sequence $r^\eps\rightarrow \theta^2=:r$ in $L^p_{loc}(\Omega_T)$.\\
This proves the first assertion of the theorem. 

Now, in order to conclude the convergence of $\mathbf{U}^\eps \rightarrow \mathbf{U}$, we need the convergence of $\xi^\eps\rightarrow \xi$, where $\xi^\eps$ satisfies \eqref{eq: xi_equation}. 

For this, we differentiate \eqref{eq: xi_equation} with respect to $x$ to obtain
\[
{(\xi_x^\eps)}_t + {(b^\epsilon \xi_x^\epsilon)}_x =0.
\]
Now, for $\delta >0$, let $\eta_\delta(\xi_x):=\sqrt{\xi_x^2+\delta^2}$. Multiplying the above equation by
$\eta_\delta'(\xi_x^\epsilon)$ and using the chain rule, one obtains
\[
\partial_t \eta_\delta(\xi_x^\epsilon)
+
\partial_x\bigl(b^\epsilon \eta_\delta(\xi_x^\epsilon)\bigr)
=
(b^\epsilon)_x\Bigl(\eta_\delta(\xi_x^\epsilon)-\xi_x^\epsilon\eta_\delta'(\xi_x^\epsilon)\Bigr).
\]
Since
\[
\eta_\delta(\xi_x)-\xi_x\eta_\delta'(\xi_x)
=
\frac{\delta^2}{\sqrt{\xi_x^2+\delta^2}}
\to 0
\qquad\text{as }\delta\rightarrow 0,
\]
we may pass formally to the limit and obtain
\[
\partial_t |\xi_x^\epsilon| + \partial_x(b^\epsilon |\xi_x^\epsilon|)=0
\]
in the sense of distributions. Using an appropriate sequence of test functions, we deduce for all $t\in [0,T]$ the relation \[
\int_{\mathbb R}|\xi_x^\epsilon(x,t)|\,dx = \int_{\mathbb R}|\xi_{x,0}(x)|\,dx = \Xi.
\]
Then, for all $t\in[0,T]$, the family ${\{\xi^\epsilon_x(\cdot,t)\}}_{\epsilon>0}$ is
uniformly bounded in $L^1(\mathbb R)$. Moreover, note that $r^\eps$ satisfies  the uniformly parabolic equation  \eqref{eq: r_equation} which leads to the  uniform bound  ${\lVert\epsilon r_x\rVert}_{L^\infty(\mathbb R)}<C$ for some uniform constant $C>0$. This implies together with $r^\eps \in [m,M]$, that the transport term \[
b^\eps = \phi(r^\epsilon)-\epsilon \frac{r_x^\epsilon}{r^\epsilon}\]
is uniformly bounded in $L^\infty(\Omega_T)$. Thus, using the evolution equation \eqref{eq: xi_equation}, we observe that
the sequence ${\{\xi^\epsilon_t(\cdot,t)\}}_{\epsilon>0}$ of time derivatives  remains uniformly bounded in $L^1(\Omega)$ for all $t\in [0, T]$. Further, in view of Lemma \ref{invariant_region}, the sequence  ${\{\xi^\epsilon(\cdot,t)\}}_{\epsilon>0}$ is uniformly bounded. Therefore, by Helly's selection theorem (see \cite{BV_compactness_book} for more details), one may extract a subsequence and find a $\xi \in L^\infty(\Omega_T)$, such that
\[
\xi^\epsilon\to \xi
\qquad\text{in }L^1_{\mathrm{loc}}(\Omega_T) \text{ as } \eps \to 0
\]
and almost everywhere in $\Omega_T$.\\
Due to the strong convergence of the sequence $(r^\epsilon, \xi^\epsilon)\to (r, \xi)$ and the (invertible) relation  $(u^\epsilon, v^\epsilon)=\bigg(\sqrt{r^\epsilon \xi^\epsilon}, \sqrt{\dfrac{r^\epsilon}{\xi^\epsilon}}\bigg)$
on the compact set
\[
[m^2,M^2]\times \left[\frac mM,\frac Mm\right],
\]
we can identify 
a function $\U= (u, v)^\top:= \bigg(\sqrt{r \xi}, \sqrt{\dfrac{r}{\xi}}\bigg)^\top$ in $L^\infty(\Omega_T)$ with 
\begin{equation}\label{stronglimit}
\U^\eps \to \U \text{ a.e.\ in } \Omega_T \text{ as } \eps \to 0.
\end{equation}

We now proceed to prove that the limit $\U$ of the sequence of the approximate solutions $\U^\eps = (u^\eps, v^\eps)^\top$ is a weak entropy solution of the hyperbolic system \eqref{eq: main_system}. Consider the sequence of smooth solutions $(u^\eps, v^\eps)$ and an arbitrary  vector-valued test function $\bm{\varphi}\in C_0^1(\Omega_T; \mathbb{R}^2)$. Then, by multiplying the system \eqref{eq: main_system_viscos_2} with $\bm{\varphi}$ and integrating over $\Omega_T$, we obtain
\begin{equation}\label{weak_form_smooth}
    \displaystyle\iint_{\Omega_T} 
    \Big( \mathbf{U}^\eps\cdot {\bm \varphi}_t
          + \mathbf{F}(\mathbf{U}^\eps)\cdot {\bm \varphi}_x
         \Big)\,dx\,dt
    + \displaystyle\int_{\mathbb{R}} \mathbf{U}_0^\eps\cdot {\bm \varphi}(\cdot, 0)\,dx = \epsilon \displaystyle\iint_{\Omega_T}  \left(\dfrac{r_x^\eps}{v^\eps}, \dfrac{r_x^\eps}{u^\eps}\right)\cdot \bm{\varphi}_x\, dx\, dt,
\end{equation}
From the strong limit $\mathbf{U}^\eps\rightarrow \mathbf{U}$ in \eqref{stronglimit}, we deduce  $\mathbf{F(U^\eps)}\rightarrow \mathbf{F(U)}$ a.e.\ in $\Omega_T$. Moreover, in view of the entropy identity \eqref{entrop_h_p_general}, there exists a constant $\tilde C>0$, independent of $\epsilon$, such that
\[
\epsilon \iint_{\Omega_T} |r_x^\eps|^2\,dx\,dt \le \tilde  C.
\]
It follows that
\[
\|\epsilon r_x^\eps\|_{L^2(\Omega_T)}^2
=
\epsilon^2\|r_x^\eps\|_{L^2(\Omega_T)}^2
\le \tilde C\epsilon \to 0.
\]
Hence, using the bounds of $(u^\eps, v^\eps)$ from lemma \eqref{invariant_region}, we obtain
\[
\epsilon \displaystyle\iint_{\Omega_T}  \left(\dfrac{r_x^\eps}{v^\eps}, \dfrac{r_x^\eps}{u^\eps}\right)\cdot \bm{\varphi}_x\, dx\, dt \to 0
\qquad\text{as }\epsilon\to0.
\]
Passing to the limit in \eqref{weak_form_smooth} and using the strong convergence of $\mathbf U^\epsilon\to \mathbf U$ and $\mathbf{U}_0^\eps\rightarrow \mathbf{U}_0$, we conclude that the limit function  $\mathbf{U}=(u, v)^\top$ is a weak solution of the system \eqref{eq: main_system}.\\
Finally, we prove that the limit function $\mathbf{U}$ is actually a weak entropy solution. To this end, we utilize Definition \ref{definition_entropy_viscous}. Fix  a  non-negative test function $\varphi\in C_0^1(\Omega_T; \mathbb{R})$. The sequence of classical solutions $\mathbf{U}^\eps$ satisfies the integral identity 
\begin{equation}\label{definition_entropy_viscous_limit}
\begin{aligned}
   \displaystyle \iint_{\Omega_T} \Big(\mathcal{E}_{k, p}(\mathbf{U}^\eps)\, \varphi_t+\mathcal{Q}_{k, p}(\mathbf{U}^\eps)\,\varphi_x\Big) \, dx\, dt&+\displaystyle\int_{\mathbb{R}} \mathcal{E}_{k, p}(\mathbf{U}_0^\eps)\,\varphi(\cdot, 0)\, dx\\
   &\geq \epsilon \displaystyle \iint_{\Omega_T} \Big(\nabla \mathcal{E}_{k, p}(\mathbf{U}^\eps) \mathbf{B(U^\eps)} \mathbf{U}_x^\epsilon\Big)\, \varphi_x \, dx\,dt. 
\end{aligned}
\end{equation}
Now for the entropy $\mathcal{E}_{k,p}$ as defined in \eqref{entropy_main_convex}, we compute directly $\nabla \mathcal{E}_{k, p}(\mathbf{U}^\eps) \mathbf{B(U^\eps)} \mathbf{U}_x^\epsilon$. First, using
\[
\mathbf{B}(\mathbf{U}^\eps)\mathbf{U}_x^\eps
=
\begin{pmatrix}
u_x^\eps+\dfrac{u^\eps}{v^\eps}v_x^\eps\\[4pt]
\dfrac{v^\eps}{u^\eps}u_x^\eps+v_x^\eps
\end{pmatrix}
=
\begin{pmatrix}
\dfrac{r_x^\eps}{v^\eps}\\[4pt]
\dfrac{r_x^\eps}{u^\eps}
\end{pmatrix},
\]
we obtain
\[
\nabla\mathcal E_{k,p}(\mathbf U^\eps)\mathbf B(\mathbf U^\eps)\mathbf U_x^\eps
=
\left(\frac{\partial_{u^\eps} \mathcal E_{k,p}(\mathbf U^\eps)}{v^\eps}
+\frac{\partial_{v^\eps} \mathcal E_{k,p}(\mathbf U^\eps)}{u^\eps}\right)r_x^\eps.
\]
Since $\mathcal E_{k,p}(u,v)=r^{-k}+\sqrt r\,\xi^p$, a direct computation shows that
\[
\frac{\partial_{u^\eps} \mathcal E_{k,p}}{v^\eps}+\frac{\partial_{v^\eps} \mathcal E_{k,p}}{u^\eps}
=
-2k(u^\eps v^\eps)^{-k-1}+(u^\eps)^{p-1/2}(v^\eps)^{-p-1/2},
\]
and therefore
\[
\nabla \mathcal E_{k,p}(\mathbf U^\eps)\mathbf B(\mathbf U^\eps)\mathbf U_x^\eps
= \big(-2k(u^\eps v^\eps)^{-k-1}+(u^\eps)^{p-1/2}(v^\eps)^{-p-1/2}\big)\,r_x^\eps.
\]
By Lemma~\ref{invariant_region}, the sequences ${\{u^\eps\}}_{\eps >0},{\{v^\eps\}}_{\eps >0},{\{r^\eps\}}_{\eps >0},{\{\xi^\eps\}}_{\eps >0}$ remain uniformly bounded, hence ${\big\{\big(-2k(u^\eps v^\eps)^{-k-1}+u^{p-1/2}v^{-p-1/2}\big)\big\}}_{\eps >0}$ is uniformly bounded. Using the entropy identity \eqref{entrop_h_p_general} again, it follows that
\[
\epsilon \iint_{\Omega_T} \Big(\nabla \mathcal{E}_{k,p}(\mathbf{U}^\eps) \mathbf{B}(\mathbf{U}^\eps) \mathbf{U}_x^\epsilon\Big)\, \varphi_x \, dx\,dt \to 0
\qquad\text{as }\eps\to0.
\]
Passing to the limit in \eqref{definition_entropy_viscous_limit} and using the strong convergence of $\mathbf U^\epsilon\to \mathbf U$ and $\mathbf{U}_0^\eps\rightarrow \mathbf{U}_0$, we conclude that
\[
\iint_{\Omega_T} \Big(\mathcal{E}_{k,p}(\mathbf{U})\, \varphi_t+\mathcal{Q}_{k,p}(\mathbf{U})\,\varphi_x\Big)\, dx\, dt
+\int_{\mathbb{R}} \mathcal{E}_{k,p}(\mathbf{U}_0)\,\varphi(\cdot, 0)\, dx
\ge 0.
\]
Thus, $\mathbf U=(u, v)^\top$ is a weak entropy solution of \eqref{eq: main_system} in the sense of Definition \ref{entropy_soln}.
\end{proof}
\section{Conclusions and future outlook}\label{sec: conclusions}
We established the existence of weak entropy solutions for a class of hyperbolic 
Keyfitz--Kranzer systems by introducing a novel approximation specifically designed for the compactness framework. The construction of the approximation is motivated by lubrication models arising in thin-film flow models, where diffusion appears only for a single component, and the system therefore exhibits a partially diffusive structure. In our analysis, however, we incorporate an additional artificial diffusion term, which results in a second-order evolution that is compatible with entropy/entropy-flux pairs of the underlying first-order hyperbolic system. This enabled us to derive the necessary compactness estimates independent of the diffusion coefficient. 

It would be interesting to investigate whether our approach can be extended to more complicated hyperbolic systems consisting of three or more equations, for e.g. the systems developed in the recent works of Barthwal et al.\ on multi-layer thin film flows \cite{barthwal2025hyperbolic, barthwal2025existence, barthwal2026generalized}. Furthermore, the invariant-region structure and the analytical framework developed in this work suggest the possibility of constructing numerical schemes that preserve the intrinsic properties, such as entropy stability and positivity preservation of the underlying system at the discrete level, which will be the content of a forthcoming work.

\section*{Acknowledgements}
Financial support by the German Research Foundation (DFG), within the framework of priority research programme—SPP 2410 Hyperbolic Balance Laws in Fluid Mechanics: Complexity, Scales, Randomness (CoScaRa) is gratefully acknowledged. P.Ö. is additionally supported by the individual grant  OE 661/4-1(520756621).
\bibliographystyle{abbrv}
\bibliography{citation}
\end{document}